\documentclass[11pt]{article}

\usepackage[title]{appendix}
\usepackage{subfiles}
\usepackage{graphicx}
\usepackage{caption}
\usepackage{subcaption}
\graphicspath{ {./images/} }
\usepackage[utf8]{inputenc}
\usepackage{amsmath}
\usepackage{amssymb, amsthm, physics,leftidx, bbm} 
\usepackage{yhmath}
\usepackage{mathrsfs}

\usepackage[maxnames=99]{biblatex}
\usepackage{mathrsfs}
\usepackage{tikz-cd}

\usepackage{enumerate}
 \usepackage{soul}
\setstcolor{red}
\usepackage[normalem]{ulem}
\usepackage{dashrule}

\allowdisplaybreaks 
\usepackage[all]{xy}
\usepackage{mathtools}
\usepackage{stmaryrd}
\usepackage[shortlabels]{enumitem}

\usepackage{tikz-cd}
\usepackage[
 letterpaper, 
 margin = 1.5 in
 ]{geometry}

\theoremstyle{definition}

\newtheorem{thm}{Theorem} 
\newtheorem{defn}{Definition}[section]	

\newtheorem{prop}{Proposition}[section]
\newtheorem{lem}{Lemma}[section]

\newtheorem{conj}{Conjecture}
\newtheorem{assumptions}{Assumption}
\newtheoremstyle{dotlessS}{}{}{}{}{\color{blue}\bfseries}{}{ }{}
\theoremstyle{dotlessS}

\DeclareMathOperator{\N}{\mathbb{N}}

\DeclareMathOperator{\C}{\mathbb{C}}
\DeclareMathOperator{\bulk}{\pmb{\mathfrak{b}}}

\DeclareMathOperator{\Ainf}{A_{\infty}} 

\DeclareMathOperator{\Z}{\mathbb{Z}}

\DeclareMathOperator{\R}{\mathbb{R}}

\DeclareMathOperator{\s}{\mathfrak{s}}
\DeclareMathOperator{\rindex}{\mathfrak{r}}

\DeclareMathOperator{\g}{\mathfrak{g}}

\DeclareMathOperator{\evaluation}{ev}

\DeclareMathOperator{\E}{\mathcal{E}}

\DeclareMathOperator{\novring}{\Lambda_{0}}

\DeclareMathOperator{\m}{\mathfrak{m}}

\DeclareMathOperator{\q}{{\mathfrak{q}}}

\DeclareMathOperator{\ld}{\mathcal{L}}

\DeclareMathOperator{\Mq}{\mathcal{M}}
\DeclareMathOperator{\kur}{\mathcal{M}}

\DeclareMathOperator{\val}{val}

\newcommand{\pair}[1]{\left\langle #1 \right\rangle }

\newcommand{\set}[2]{\left\{ #1 \, \middle|\, #2\right\}}
\newcommand{\lrp}[1]{\left(#1\right)}

\newcommand{\littletaller}{\mathchoice{\vphantom{\big|}}{}{}{}}
\newcommand\restr[2]{{% we make the whole thing an ordinary symbol
  \left.\kern-\nulldelimiterspace % automatically resize the bar with \right
  #1 % the function
  \littletaller % pretend it's a little taller at normal size
  \right|_{#2} % this is the delimiter
  }}

\usepackage{imakeidx}
\makeindex[intoc]
\usepackage[hidelinks]{hyperref}
\def\equationautorefname~#1\null{(#1)\null}

\usepackage{xurl}
\hypersetup{breaklinks=true}

\usepackage{lipsum}

\begin{document} 

\title{Moment Lagrangian correspondences are unobstructed after bulk deformation}
\author{Yao Xiao}
\date{\today}

\maketitle 
\newcommand{\Addresses}{{% additional braces for segregating \footnotesize
  \bigskip
  \footnotesize

  \textsc{Department of Mathematics, Stony Brook University,
    Stony Brook, NY 11794}\par\nopagebreak
  \textit{E-mail address}: \texttt{yao.xiao@stonybrook.edu}

}}

\begin{abstract}
 We prove that the Lagrangian correspondences induced by the symplectic reduction maps at free zero level sets of the moment maps are unobstructed after bulk deformation, assuming the existence of certain equivariant Kuranishi structures and compatible equivariant CF-perturbations on the moduli spaces of pseudoholomorphic discs. 
\end{abstract}
\tableofcontents
\section{Introduction}
\label{section Introduction}

A moment Lagrangian correspondence serves as a bridge between a Hamiltonian $G$-manifold $Y$ and its symplectic reduction $Y\sslash G$. In particular, an unobstructed moment Lagrangian correspondence induces a functor from the Fukaya category of $Y$ to the Fukaya category of $Y\sslash G$. (See  \cite{fukaya2023unobstructed} and \cite{MWW}.)  
Recently, Lau-Leung-Li \cite{LLL} exploited such relations and
applied Fukaya's $A_{\infty}$ tri-module structure (see \cite{fukaya2023unobstructed}) to prove Teleman's conjecture.  They have also proved the (weak) unobstructedness of these moment Lagrangian correspondences under some topological assumptions. 
By taking the $G$-action into account, we prove that moment Lagrangian correspondences are unobstructed after bulk deformation, whenever we can equip the moduli spaces of pseudoholomorphic discs with $G$-equivariant Kuranishi structures and compatible $G$-equivariant CF-perturbations such that $G$ acts freely on each Kuranishi chart. These notions related to $G$-equivariant Kuranishi structures have been discussed in \cite{FukayaLieGroupoids} and \cite{xiao2023equivariant}. 
The existence of such equivariant Kuranishi structures and CF-perturbations has been shown, for example, in the case of the moduli spaces with boundary on Lagrangian torus fibers in compact symplectic toric manifolds in \cite{III}.  
And we expect such equivariant Kuranishi structures to exist whenever $G$ acts freely on $\mu^{-1}(\xi)$. 
We refer the reader to \cite{SS}, \cite{HLS1}, \cite{HLS2}, \cite{Cazassus}, \cite{FutakiSanda}, \cite{KLZ}, \cite{HKLZ}, and \cite{LLL} for other approaches to equivariant Lagrangian Floer theory. 

We first recall the definition of a moment Lagrangian correspondence. 
Let $(Y,\omega_Y,G,\mu)$ be a Hamiltonian $G$-manifold consisting of the following data. 
\begin{itemize}
    \item $(Y,\omega_Y)$ is a compact symplectic manifold. 
    \item $G$ is a compact connected Lie group acting on $(Y,\omega_Y)$ in a Hamiltonian fashion. Let $\g$ be the Lie algebra of $G$ and $\g^*$ be its dual. 
    \item $\mu: Y\to \g^*$ is a moment map of the $G$-action. 
\end{itemize}
Suppose $G$ acts on $\mu^{-1}(0)$ freely. By the theorem of Marsden-Weinstein and Meyer, there exist a symplectic reduction map
$\pi: \mu^{-1}(0)\to  \mu^{-1}(0)/G =: Y\sslash G$ and an induced symplectic form $ \omega_{red}$ on $Y\sslash  G$. 
We will call the Lagrangian submanifold
\begin{equation}
\label{moment Lagrangian correspondence}
L = \set{ (p, [p])\in  Y^-\times Y\sslash G}{ p\in  \mu^{-1}(0),\quad \pi(p)=[p] }      
\end{equation} 
 of $\lrp{Y^-\times Y\sslash G, -\omega_Y\oplus  \omega_{red}}$ 
 the \textbf{moment Lagrangian correspondence} induced by this symplectic reduction.  
Denote the inclusion map by 
\begin{equation}\label{inclusion map}
    \iota : L \to Y^-\times Y\sslash G.
\end{equation}
 
We refer the interested reader to \cite{abouzaid2022functoriality} and  \cite{fukaya2023unobstructed} to learn more about Lagrangian correspondences. 

Let $G$ act on $Y^-\times Y\sslash G$ such that it acts trivially on the second factor $ Y\sslash G$ and acts on the first factor $Y^-$ by the original Hamiltonian action. 
We denote $Y^-\times Y\sslash G$ by $X$ and $-\omega_Y\oplus  \omega_{red}$ by $\omega$. Let $J$ be a $G$-invariant $\omega$-compatible almost complex structure on $X$. And we equip $L$ with a relatively spin structure. 
Moreover, 
for each $k,l\in \N$ and $\beta\in \pi_2(X,L)$, define the moduli space
$\Mq_{k+1,l}(L,J,\beta)$ by
\begin{equation}
\label{moduli space of discs notation}
\set{ (\Sigma,j_{\Sigma}, \vec{z},\vec{w},u)}{ 
\begin{aligned}
& \Sigma \text{ is a genus $0$ nodal Riemann surface 
with }\\
& \text{connected boundary and complex structure } j_{\Sigma}; \\
& u:(\Sigma,\partial \Sigma)\to (X,L) 
  \text{ is smooth}; \\
& du\circ j_{\Sigma} =J\circ du; 
\quad [u] = \beta \in \pi_2(X,L); \\
& (\Sigma,j_{\Sigma}, \vec{z}, \vec{w},u) \text{ is stable};  \quad   E(u) 
<\infty\\
& \vec{z} =\left(z_0,z_1,\ldots, z_k\right)  \in (\partial \Sigma)^{k+1}, 
\text{where the $z_i$ are } \\
& \text{distinct non-nodal boundary marked points and the} \\ & \text{enumeration is in counterclockwise order along }\partial \Sigma; \\
& \vec{w} =\left(w_1,\ldots, w_l\right)  \in (\mathring{\Sigma})^{l} \text{ are distinct non-nodal }\\
& \text{interior marked points }
\end{aligned}
} \Big/ \sim ,
\end{equation} 
where $(\Sigma,j_{\Sigma}, \vec{z}, \vec{w},u) \sim (\Sigma',j_{\Sigma'}, \vec{z}^{'},\vec{w}^{'},u')$ if and only if there exists a biholomorphism $\varphi: (\Sigma,j_{\Sigma}) \to (\Sigma',j_{\Sigma'})$ such that $u'\circ \varphi = u$, $\varphi (z_i) = z_i'$ for all $0\leq i \leq k$, and $\varphi(w_j) = w_j'$ for all $1\leq j\leq l$. 
We denote the evaluation map at the $i$-th boundary marked point by
\[ \evaluation_{i,(k+1,l,\beta)}: \Mq_{k+1,l}(L,J,\beta)\to L, \quad [\Sigma,j_{\Sigma}, \vec{z},\vec{w},u]\mapsto u(z_i) \qquad \forall 0\leq i\leq k, \]
and we denote the evaluation map at the $j$-th interior marked point by
\[ \evaluation_{(k+1,l,\beta)}^j: \Mq_{k+1,l}(L,J,\beta)\to L, \quad  [\Sigma,j_{\Sigma}, \vec{z},\vec{w},u] \mapsto u(w_j) \qquad \forall 1\leq j\leq l. \]

The main result of the paper is Theorem \ref{main theorem}. 
\begin{thm}[Main theorem]\label{main theorem}
The Lagrangian submanifold 
\[ L\subset \lrp{Y^-\times Y\sslash G , -\omega_Y \oplus \omega_{red}} \] defined by \eqref{moment Lagrangian correspondence} is unobstructed after bulk deformation under Assumption \ref{assumptions}, which is stated below. 
\end{thm}

\begin{assumptions}\label{assumptions}

For each $k,l\in \N$ and $\beta\in \pi_2(X,L)$, we assume that the following holds. 
    \begin{enumerate}[i)]
        \item  $\Mq_{k+1,l}(L,J,\beta)$  
        has a $G$-equivariant Kuranishi structure, and $G$ acts freely on each Kuranishi chart. 
        \item The evaluation map at every (interior or boundary) marked point is strongly smooth and is $G$-equivariant on each chart. 
        \item \label{G-invariant CF-perturbations}
        There is a compatible $G$-equivariant system of CF-perturbations $\widehat{\mathcal{S}}$ such that the Thom forms in the CF-perturbation data \eqref{equation CF perturbation representative} are $G$-basic.
        \item Moreover, the equivariant Kuranishi structures and equivariant CF-purturbations are compatible with
{ \fontsize{9.5}{11}
\begin{align*}
   & \partial \Mq_{k+1,l}(L,J,\beta) = \\ 
    & \bigcup_{\substack{k_1,k_2, l_1,l_2\geq 0 \\ k_1+k_2=k+1\\l_1+l_2=l}} 
     \bigcup_{\substack{
    \beta_1,\beta_2 \in \pi_2(X,L)\\ \beta_1+\beta_2=\beta}} 
     \bigcup_{j=1}^{k_2} 
      \Mq_{k_1+1,l_1}(L,J,\beta_1) 
   _{\evaluation_{0, (k_1+1,l_1,\beta_1)}}{\times} _{\evaluation_{j, (k_2+1,l_2,\beta_2)}} \Mq_{k_2+1,l_2} (L,J,\beta_2). 
\end{align*}
}
        \item  The evaluation map $\evaluation_{0, (k+1,l,\beta)}: \Mq_{k+1,l}(L,J,\beta)\to L$ at the zero-th boundary marked point is strongly submersive with respect to $\widehat{\mathcal{S}}$. 
       
    \end{enumerate}
\end{assumptions} 

We expect such assumptions to hold when the restriction of the Hamiltonian $G$-action on a symplectic manifold $X$ to a compact $G$-invariant Lagrangian submanifold $L$ is free.  
We will explain more on the notions appearing in the assumption in Section \ref{section Proof of the key lemma}.

We review the bulk deformation construction in Section \ref{section Bulk deformation}. 
Assuming Lemma \ref{key lemma}, we prove the main theorem in Section \ref{section Proof of the main theorem}. 
And we prove Lemma \ref{key lemma} in Section \ref{subsection Proof of key lemma }   after introducing some notions related to $G$-equivariant Kuranishi structures in Section \ref{section G-equivariant Kuranishi data} and reviewing some equivariant de Rham theory in Section \ref{section Existence of basic Thom forms for locally free actions}. 
In the end, we state a related conjecture in Section \ref{section Conjecture on the quantum Kirwan map}. 

\subsection*{Acknowledgement} 
I would like to thank my advisor, Prof.\@ Kenji Fukaya, for suggesting this problem to me and for helpful discussions. 
This paper is part of the author's PhD thesis at Stony Brook University and was partially supported by Simons Foundation International, LTD.
\section{Bulk deformation}
\label{section Bulk deformation} 
We assume the setup and notations in \S \ref{section Introduction}. We take the \textbf{universal Novikov ring}
	\begin{equation}
	\label{universal Novikov ring}
	\Lambda_{0,nov}
	=
	\set{ \sum_{i\in \N} a_iT^{\lambda_i} e^{n_i}
	}
	{ a_i\in \C, \lambda_i\in \R_{\geq 0}, n_i\in \Z,  \lim_{i\to \infty} \lambda_i = \infty
	}	  
	\end{equation} 
to be our coefficient ring. 
And let 
\begin{equation}
	\label{positive Novikov ring}
	\Lambda_{0,nov}^+
	=
	\set{ \sum_{i\in \N} a_iT^{\lambda_i} e^{n_i} \in \Lambda_{0,nov}
	}
	{ \lambda_i >0 \quad \forall i
	}. 
\end{equation} 
For each $\beta\in\pi_2(X,L)$, we denote its symplectic area by $\omega(\beta)$ and its Maslov index by $I_{\mu}(\beta)$. 
$\Lambda_{0,nov}$ comes with a valuation function $\val: \Lambda_{0,nov} \to \R\cup \{\infty\}$ given by the following.
\begin{equation}
\label{valuation universal novikov field}
y=\sum\limits_{i\in \N }  a_i T^{\lambda_i} e^{n_i} 
\mapsto 
\val (y) := \begin{cases}
\min \{ \lambda_i\mid i\in \N, a_i\ne 0\} \quad &\text{if }y \ne 0 \\
\infty \quad &\text{if }y =0
\end{cases}     
\end{equation}

\subsection{The \texorpdfstring{$A_{\infty}$}{A-infinity} algebra associated to a Lagrangian submanifold}
The $A_{\infty}$ algebra 
\[ \lrp{ \Omega(L,\Lambda_{0,nov}), \{\m_k\}_{k\in \N}} \]
associated to $L$ is given by the following data. 
\begin{itemize}
    \item $\Omega(L,\Lambda_{0,nov}) = \Omega(L)\widehat{\otimes}_{\R} \Lambda_{0,nov}$, where $\Omega(L)$ denotes the de Rham complex of $L$, and $\widehat{\otimes}$ denotes completion of the tensor product with respect to the $T$-adic topology. We specify that $\Omega(L,\Lambda_{0,nov})^{\otimes 0} = \Lambda_{0,nov}$, and the elements of
    $\Omega(L,\Lambda_{0,nov})[1]$ are the elements of $\Omega(L,\Lambda_{0,nov})$ with degree shifted down by $1$.  
    
    \item For each $k\in \N$, the $A_{\infty}$ operator  \[ \m_k = \sum_{\beta\in \pi_2(X,L)} \m_{k,\beta}T^{ \omega(\beta) } e^{\frac{I_\mu(\beta)}{2}}: \left(\Omega(L,\Lambda_{0,nov})[1]\right)^{\otimes k} \to \Omega(L,\Lambda_{0,nov})[1] \] 
    is defined by the following. 
    Let $x_1,\ldots, x_k\in \Omega(L,\Lambda_{0,nov})$.
    For $\beta=0$,  
    \begin{equation}
\label{Ainf operator beta = 0}
 \begin{dcases}
\m_{0,\beta=0}(1)= 0 \quad &  \\
\m_{1,\beta=0}(x_1)= dx_1,
\text{ where $d$ is the de Rham differential,} &   \\
\m_{2,\beta=0} (x_1\otimes x_2) = (-1)^{\deg x_1}  x_1 \wedge  x_2 \quad &   \\
\m_{k,\beta=0}=0 \qquad  \forall k\geq 3. \quad &
\end{dcases} 
\end{equation}
For $\beta\ne 0$, 
define
\begin{equation} \label{Ainf operator beta non-zero k=0}
\m_{0,\beta}(1) = (\evaluation_{0, (1,0,\beta)})_!(1),   
\end{equation}
and, for all $k \geq 1$,
\begin{align*}
 & \m_{k,\beta}(x_1\otimes \cdots \otimes x_k) \nonumber \\
 = & (-1)^{1+\sum_{i=1}^k (i\deg x_i +1)} (\evaluation_{0, (k+1,0,\beta)})_!\left(   \evaluation_{1, (k+1,0,\beta)}^*x_1\wedge \cdots \wedge \evaluation_{k, (k+1,0,\beta)}^*x_k\right).      
\end{align*}
\label{Ainf operator beta non-zero} 

\end{itemize}
We note that the signs in the definition of the $\m_k$ follow that in \cite{kurbook} Chapter 22. 

\subsection{Bulk deformation of the \texorpdfstring{$A_{\infty}$}{A-infinity} algebras}
Consider the $\Lambda_{0,nov}$-module homomorphisms of the form
\[ \mathfrak{q}_{l,k}: \left(\Omega(X,\Lambda_{0,nov})[2]\right)^{\otimes l}  \otimes \left(\Omega(L,\Lambda_{0,nov})[1] \right)^{\otimes k}\to \Omega(L,\Lambda_{0,nov})[1], \]
\[ \mathfrak{q}_{l,k} = \sum_{\beta\in \pi_2(X,L)} \mathfrak{q}_{l,k,\beta}T^{ \omega(\beta) } e^{\frac{I_\mu(\beta)}{2}}, \]
where $\q_{l,k,\beta}: \Omega(X,\Lambda_{0,nov})^{\otimes l}\otimes \Omega(L,\Lambda_{0,nov})^{\otimes k} \to \Omega(L,\Lambda_{0,nov}) $ is defined by, for all $ x_1,\ldots, x_k\in \Omega(L,\Lambda_{0,nov})$ and all $ y_1,\ldots, y_l\in \Omega(X,\Lambda_{0,nov})$,  
\begin{align}\label{equation qlkbeta}
\begin{dcases}
& \mathfrak{q}_{0,k,\beta} (x_1\otimes \cdots \otimes x_k) 
= \m_{k,\beta} (x_1\otimes \cdots \otimes x_k)  \\
& \mathfrak{q}_{1,0,\beta=0} (y_1)  = (-1)^{\dagger}\iota^*y_1, \quad \text{ where $\iota: L\hookrightarrow X$ is the inclusion map } \\
& \mathfrak{q}_{1,0,\beta} (y_1)   = 0 \quad \text{ if }    \beta \ne 0 \\ 
& \mathfrak{q}_{l,k,\beta}(y_1\otimes \cdots \otimes y_l\otimes x_1\otimes  \cdots \otimes x_k) = \\
&  (-1)^{\ddag}\frac{1}{l!} (\evaluation_{0, (k+1,l,\beta)})_!\Big( (\evaluation_{ (k+1,l,\beta)}^1)^*y_1\wedge \cdots \wedge  (\evaluation_{ (k+1,l,\beta)}^l)^*y_l\wedge  \\
& \quad \evaluation_{1, (k+1,l,\beta)}^*x_1\wedge \cdots \wedge \evaluation_{k, (k+1,l,\beta)}^*x_k\Big) \\ 
& \;\text{ if } l\ne 0, k\ne 0, \text{ and } (l,k)\ne (1,0). 
\end{dcases}  
\end{align}
Here $\dagger$ is an integer depending on the degree of the differential form $y$, and $\ddag$ is an integer depending on the degrees of the differential forms $y_1,\ldots, y_l, x_1,\ldots, x_k$. 

Let $\bulk\in \Omega^{even}(X,\Lambda_{0,nov})$ and $b\in\Omega^{odd}(L,\Lambda_{0,nov}) $.  
The bulk-deformed $A_{\infty}$ operators are defined by 
\begin{align*}
\label{equation bulk mk}
    & \m_k^{\bulk,b}(x_1\otimes \cdots \otimes x_k)\\
    = & \sum_{\beta\in \pi_2(X,L)}\sum_{l\geq 0} \sum_{r_0,\ldots, r_k\geq 0}\q_{l,r_0+\cdots +r_k+k,\beta}(\bulk^{\otimes l} \otimes b^{r_0}\otimes x_1\otimes b^{r_1}\otimes \cdots \otimes x_k\otimes b^{r_k})T^{\omega(\beta)}e^{\frac{I_{\mu}(\beta)}{2}}. 
\end{align*}
In particular, for any $\bulk\in \Omega^{even}(X,\Lambda_{0,nov})$ and $b\in\Omega^{odd}(L,\Lambda_{0,nov}) $, 
\begin{equation}
\label{equation bulk m0}
\m_{0}^{ \bulk, b}(1) = \sum_{\substack{l,k\in \N \\  \beta \in \pi_2(X,L)} } \q_{l,k,\beta} (\bulk^{\otimes l} \otimes b^{\otimes k})T^{\omega(\beta)}e^{\frac{I_{\mu}(\beta)}{2}}.    
\end{equation} 
We say $L$ is \textbf{unobstructed after bulk deformation} by $\bulk, b$ if 
\[ \m_{0}^{ \bulk, b}(1) =0. \]
Since the bulk-deformed $\Ainf$ operators $\m_k^{\bulk,b}$ 
still satisfy the $\Ainf$ relations, 
there will be no obstruction in defining the Floer cohomology of $L$ by 
\[HF(L,L,\novring) : = H^* \lrp{\Omega(L,\Lambda_{0,nov}), \m_1^{ \bulk, b}}. \]
We refer the readers to \cite{fooobook1}, \cite{FOOObook2}, \cite{II}, and \cite{III} for more details on bulk deformation theory.

\section{Proof of  Theorem \ref{main theorem}}
\label{section Proof of the main theorem}
The main consequence of Assumption \ref{assumptions} is Lemma \ref{key lemma}, the proof of which will be postponed to Section \ref{section Proof of the key lemma}. 

To state Lemma \ref{key lemma}, we recall the following definitions. 
For Definition \ref{defn Fundamental vector fields} and Definition \ref{defn basic forms}, let $G$ be a compact connected Lie group acting smoothly on a smooth manifold $M$. Let $\g$ be the Lie algebra of $G$. 
\begin{defn}[Fundamental vector fields]
\label{defn Fundamental vector fields}
There is a Lie algebra homomorphism $\sigma: \g \to \Gamma(TM)$, which assigns to every $\zeta\in \g$ its \textbf{fundamental vector field}\index{fundamental vector field} $\underline{\zeta}$ on $M$, as follows. 
\begin{equation}
\label{equation fundamentall vector field Lie algebra homomorphism}
\sigma(\zeta)_p : =\underline{\zeta}(p) := \frac{d}{dt}\bigg|_{t=0} (e^{-t\zeta}\cdot p) \qquad p\in M.
\end{equation}
\end{defn} 
 
\begin{defn}[$G$-basic forms on a $G$-manifold] 
\label{defn basic forms}
 A differential form $\alpha\in \Omega(M,R)$ on $M$ with coefficients in a commutative ring $R$ is \textbf{$G$-basic} if, for $X\in \g$, the following holds. 
    \begin{enumerate}[i)]
        \item $\alpha$ is \textbf{$G$-invariant}:  $\mathcal{L}_{\underline{\zeta}} \alpha =0$.   
        \item $\alpha$ is \textbf{$G$-horizontal}:  $\iota_{\underline{\zeta}} \alpha =0$.  
    \end{enumerate}
  
    We can similarly define $G$-basic forms on a Kuranishi space $\mathcal{M}$ if each of its Kuranishi charts admits a $G$-action.  
\end{defn}
% We remark that, if $G$ is not connected, the $G$-invariance condition may fail to agree with invariance under pullbacks by the diffeomorphisms induced by the action.

The following lemma is well-known in equivariant de Rham theory. See \cite{GHVII} Chapter IV \S 2 Proposition III or \cite{Tu} Theorem 12.5 for instance. 
We will use it for proving Theorem \ref{main theorem}. 
\begin{lem}[Basic forms on principal bundles are pullbacks]
\label{Lemma Basic forms of principal bundles are pullbacks}   Let $G$ be a compact connected Lie group. Let $\pi: P\to B$ be a smooth principal $G$-bundle.  
Then $\pi^*: \Omega(B) \to \Omega_{bas}(P)$ is an isomorphism. 
\end{lem}

\begin{proof}
Since $\pi$ is a surjective submersion, $ \pi^*$ is injective. 
We now show $\pi^*\Omega(B) = \Omega_{bas}(P)$. 

Suppose $\eta = \pi^*\beta$ for some $\beta\in \Omega(B)$. 
Then, for any $g\in G$ and its induced diffeomorphism $\varphi_g:P\to P$, we have 
$\varphi_g^* \eta = \varphi_g^*\pi^*\beta = (\pi \circ \varphi_g)^*\beta = \eta$. 
Since $G$ is connected, this is equivalent to $\mathcal{L}_{\underline{\zeta}}\eta =0$ for all $\zeta\in \g$, showing that $\eta$ is $G$-invariant. 
Moreover, 
$d\pi \circ \underline{\zeta} =0$
for all $\zeta\in \g$. 
Hence, 
$\iota_{\underline{\zeta}} (\pi^*\beta) = 0$ for all $\zeta\in \g$. 
This shows that $\pi^*\Omega(B) \subset \Omega_{bas}(P)$. 

Suppose $\eta \in \Omega(P)$ is $G$-horizontal and $G$-invariant. 
Let $\left \{ U_{\alpha}\times G \xrightarrow{\psi_{\alpha}} \pi^{-1}(U_{\alpha})\right\} $ be a trivialization of $\pi$. Since $\psi_{\alpha}^*$ commutes with $\mathcal{L}_{\underline{\zeta}}$ and $\iota_{\underline{\zeta}}$ for all $\zeta\in \g$, the form
$\psi_{\alpha}^*(\eta\eval_{\pi^{-1}(U_i)})$ is also horizontal and invariant. 
Thus, there exists a unique $\beta_{\alpha}\in \Omega(U_{\alpha})$ such that $ \beta_{\alpha}\otimes 1  = \psi_{\alpha}^*(\eta\eval_{\pi^{-1}(U_{\alpha})})$. 
Then $(\beta_{\alpha})_x = (\beta_{\alpha'})_x$ if $x\in U_{\alpha}\cap U_{\alpha'}$, and we can define $\beta\in \Omega(B)$ by $\beta_x = (\beta_{\alpha})_x$ for $x\in U_{\alpha}$. Hence, we have $\pi^*\beta = \eta$.  
\end{proof}

The key lemma in proving Theorem \ref{main theorem} is the following. We postpone its proof to Section \ref{section Proof of the key lemma}. 
\begin{lem}[Key lemma]\label{key lemma}
Suppose Assumption \ref{assumptions} holds. 
For any $l,k \in \N$, 
the map 
    \[ \mathfrak{q}_{l,k,\beta}: \Omega(X,\Lambda_{0,nov})^{\otimes l}  \otimes \Omega(L,\Lambda_{0,nov})^{\otimes k} \to \Omega(L,\Lambda_{0,nov}) \] 
    defined in \eqref{equation qlkbeta}  maps an element of  $\Omega_{bas}(X,\Lambda_{0,nov})^{\otimes l}  \otimes \Omega_{bas}(L,\Lambda_{0,nov})^{\otimes k} $ to an element of $\Omega_{bas}(L,\Lambda_{0,nov})$. 
\end{lem}

The proof of Theorem \ref{main theorem} will be based on an induction on the monoid
 \begin{equation}
 \label{submonoid}
 \Gamma = \set{ \lrp{\omega(\beta), \frac{I_\mu(\beta)}{2}  } \in \R_{\geq 0}\times \Z }{   \beta \in \pi_2(X,L), \Mq(L,J,\beta)\ne \emptyset  }.       
 \end{equation}
 Consider the lexicographic order on $\Gamma$ given by the following. 
 Let $(\lambda , n),(\lambda ', n')\in \Gamma$.  
 \begin{enumerate}[1)]
     \item $(\lambda , n)=(\lambda ', n') $ if and only if $\lambda = \lambda', n=n'$. 
     \item $(\lambda , n)<(\lambda ', n') $ if one of the following holds. \begin{enumerate}
         \item $\lambda <\lambda'$  
         \item 
         $\lambda =\lambda'$, $n<n'$. 
    \end{enumerate} 
 \end{enumerate} 

We may renumber the elements of $\Gamma$ as follows. 
\[ \Gamma  = \big \{ (\lambda_i, n_{i,j}) \in \R_{\geq 0}\times \Z \mid i=0,1,\ldots,\quad  0\leq j\leq l_i \big \} \]
so that $\lambda_i< \lambda_{i+1}$ for all $i\geq 0$ and 
$n_{i,j}< n_{i,j+1}$ for all $1\leq j \leq l_{i-1}$. 

\begin{proof}[Proof of Theorem \ref{main theorem}]
We want to construct 
\begin{equation}\label{bulk form}
 \bulk^{(i)} = \sum_{i'=0}^i \sum_{j'=1}^{l_{i'}} \bulk_{i',j'}T^{\lambda_{i'}} e^{n_{i',j'}}  , \qquad  b^{(i)} = \sum_{i'=0}^i \sum_{j'=1}^{l_{i'}} b_{i',j'}T^{\lambda_{i'}} e^{n_{i',j'}} 
\end{equation}
such that the $\bulk_{i',j'}, b_{i',j'}$ are $G$-basic forms on $X,L$, respectively, and the terms of $\m_0^{\bulk^{(i)},b^{(i)}}(1)$ with valuation less than or equal to $\lambda_i$ vanish in the sense that 
\begin{equation}\label{inductive hypothesis}
\m_{0}^{ \bulk^{( i )}  ,  b^{(i)}}(1) \equiv 0 \mod T^{\lambda_{i}} \Lambda_{0,nov}^+  ,    
\end{equation} 
by induction on $i$. 
\\
\noindent 
Let 
\[ \bulk^{( 0 ) } =0 , \quad b^{(0)}=0. \]
Then 
\[ \m_0^{\bulk^{( 0 ) }, b^{( 0 ) }} (1)\equiv \q_{l=0,k=0,\beta=0}(1) \equiv 0 \mod \Lambda_{0,nov}^+ . 
\] 
Assume that we have constructed   
\[  \bulk^{(i)} = \sum_{i'=0}^i \sum_{j'=1}^{l_{i'}} \bulk_{i',j'}T^{\lambda_{i'}} e^{n_{i',j'}} , \quad 
  b^{(i)} = \sum_{i'=0}^i \sum_{j'=1}^{l_{i'}} b_{i',j'}T^{\lambda_{i'}} e^{n_{i',j'}} \]  
such that the $\bulk_{i',j'}, b_{i',j'}$ are $G$-basic forms and  
\begin{equation}
\label{induction hypothesis}
\m_{0}^{ \bulk^{( i )}  ,  b^{(i)} }(1) \equiv 0 \mod T^{\lambda_{i}} \Lambda_{0,nov}^+.     
\end{equation}

We want to construct 
\[  \bulk^{(i+1)} = \sum_{i'=0}^{i+1} \sum_{j'=1}^{l_{i'}} \bulk_{i',j'}T^{\lambda_{i'}} e^{n_{i',j'}}, \qquad  
b^{(i+1)} = \sum_{i'=0}^{i+1} \sum_{j'=1}^{l_{i'}} b_{i',j'}T^{\lambda_{i'}} e^{n_{i',j'}},   \]  
such that \eqref{induction hypothesis} holds with $i$ replaced by $i+1$. 
We note that if $\q_{l,k,\beta}T^{\omega(\beta)}e^{\frac{I_{\mu}(\beta)}{2}}$ 
\begin{itemize}
    \item either takes a tensor product of more than one term with positive valuation, at least one of which takes the form $\bulk_{i+1, \bullet}T^{\lambda_{i+1}}e^{n_{i+1,\bullet}}$ or $b_{i+1, \bullet}T^{\lambda_{i+1}}e^{n_{i+1,\bullet}}$ 
    \item or has $\beta\ne 0$ and takes exactly one element of the form $\bulk_{i+1, \bullet}$ and $b_{i+1, \bullet}$, 
\end{itemize}
then since the operators $\q_{l,k,\beta}$ are filtration-preserving, the resulting term will have valuation strictly higher than $\lambda_{i+1}$ and thus be $0 \mod T^{\lambda_{i+1}} \Lambda_{0,nov}^+$. 
Therefore, the contributions of $\bulk^{(i+1)}-\bulk^{(i)}$ and $b^{(i+1)}-b^{(i)}$ to 
$\m_0^{\bulk^{(i+1)}, b ^{(i+1)}} (1) \mod T^{\lambda_{i+1}}\Lambda_{0,nov}^+$
are of the form
\[ \q_{1,0,\beta=0}(\bulk_{i+1, \bullet})T^{ \lambda_{i+1}}e^{n_{i+1,\bullet}}, \quad  
\q_{0,1,\beta=0}(\bulk_{i+1, \bullet})T^{ \lambda_{i+1}}e^{n_{i+1,\bullet}} . 
\]
The other contributions must come from $\bulk^{(i)}$ and $b^{(i)}$. 
Since we know their contributions to the terms of $\m_0^{\bulk^{(i+1)}, b ^{(i+1)}} (1)$ with valuation less than or equal to $\lambda_i$ vanish, their contributions to the terms in $\m_0^{\bulk^{(i+1)}, b ^{(i+1)}} (1) \mod T^{\lambda_{i+1}}\Lambda_{0,nov}^+$ have to be exactly of valuation $T^{\lambda_{i+1}}$.

Let $o_{i+1,j}$ be the coefficient of $T^{\lambda_{i+1}}e^{n_{i+1,j}}$ in $\m_{0}^{\bulk^{(i)},b^{(i)}}(1)$. 
By the above argument, 
\begin{align*} 
& \m_0^{\bulk^{(i+1)}, b ^{(i+1)}} (1)\\ 
\equiv  
&  
\sum_{j=1}^{l_{i+1}} \Big(o_{i+1,j}
% \equiv  & \left(
% \sum_{ 
% \substack{ i_1,\ldots, i_1,\ldots, i_l , i_1',i_2'\ldots i_{l}', \beta \\
%   \lambda_{i_1}+\cdots + \lambda_{i_l} + \lambda_{i_1'}+\cdots + \lambda_{i_l'}+ \omega(\beta)=\lambda_{i+1}
% }  }
% \sum_{\substack{j_1,\ldots, j_l,\\ j_1',\ldots, j_l' }}  \q_{l,k,\beta}(\bulk_{{i_1,j_1}}\otimes \cdots \bulk_{{i_l,j_l}} \otimes b_{{i_1',j_1'}}\otimes \cdots b_{{i_l',j_l'}}) e^{n_{i_1,j_1}+\cdots n_{i_l,j_l} + n_{i_1',j_1'}+\cdots n_{i_l',j_l'}}\right)T^{\lambda_{i+1}} 
% \\
%  & + \left(\sum_{\omega(\beta)=\lambda_{i+1}}\m_{0,\beta}(1)e^{\mu(\beta)/2} \right)T^{\lambda_{i+1}}   
 % \\
 % &  
 % \\
 % & 
+  \q_{1,0,\beta=0}(\bulk_{i+1,j})  
+  \q_{0,1,\beta=0}(b_{i+1,j})  \Big) 
T^{\lambda_{i+1}}e^{n_{i+1,j}} 
 \\
 \equiv  & 
\sum_{j=1}^{l_{i+1}} 
\left( 
o_{i+1,j} 
 \pm  \iota^*(\bulk_{i+1,j}) 
 + d(b_{i+1,j}) \right)T^{\lambda_{i+1}}e^{n_{i+1,j}}
 \mod T^{\lambda_{i+1}}\Lambda_{0,nov}^+ . 
\end{align*}

Therefore, we only need to find $ \bulk_{i+1,j}, b_{i+1,j}$ such that 
\begin{equation}\label{obstruction equation}
 o_{i+1,j} \pm \iota^*(\bulk_{i+1,j}) + d(b_{i+1,j})=0.    
\end{equation} 
By Lemma \ref{key lemma}, 
we have 
\[  o_{i+1,j} \in \Omega_{bas}^{\bullet}(L, \Lambda_{0,nov}).  \]
We consider the maps 
 
\[\scalebox{0.87}{$\Omega_{bas}^{\bullet}(L, \Lambda_{0,nov})  \xrightarrow[\ref{map 1}]{ (\pi_{L/ G}^*)^{-1}}\Omega^{\bullet}(L/G, \Lambda_{0,nov}) \xrightarrow[\ref{map 2}]{ \Delta_{Y\sslash G}^*} \Omega^{\bullet}( Y\sslash G, \Lambda_{0,nov})\xrightarrow[\ref{map 3}]{\pi_{Y\sslash G}^*} \Omega^{\bullet}(Y^-\times Y\sslash G, \Lambda_{0,nov})  $}
\]
given by the following. 
\begin{enumerate}[1)]
    \item \label{map 1}
    Since $G$ acts on $L$ freely, $L\to L/G$ is a principal $G$-bundle. 
    Thus, by Lemma \ref{Lemma Basic forms of principal bundles are pullbacks}, there exists an isomorphism 
\[ \pi_{L/G}^*: \Omega^{\bullet}(L/G, \Lambda_{0,nov})\xrightarrow{\cong} \Omega_{bas}^{\bullet}(L, \Lambda_{0,nov}). \]
\item \label{map 2} 
$\Delta_{Y\sslash G}^*$ is the pullback map induced by the diagonal map
\[ \Delta_{ Y\sslash G}: Y\sslash G \to  Y\sslash G \times Y\sslash G = L/G, \qquad [p]\mapsto ([p],[p]). \] 
\item \label{map 3}
$\pi_{Y\sslash G}^*$ is the pullback map induced by the projection map $\pi_{Y\sslash G}: Y^-\times Y\sslash G \to  Y\sslash G $. 
\end{enumerate} 
Let 
\[ \bulk_{i+1,j} = \mp \pi_{Y\sslash G}^*\circ\Delta_{Y\sslash G}^*\circ (\pi_{L/ G}^*)^{-1} (o_{i+1,j}), \qquad b_{i+1,j}=0, \] 
where we use $-$ if the sign in \eqref{obstruction equation} is $+$, and we use $+$ if the sign in \eqref{obstruction equation} is $-$. 
It satisfies 
\begin{align*} 
    o_{i+1,j} \pm  \iota^*\bulk_{i+1,j}+ d(b_{i+1,j})  
    & =  o_{i+1,j} \pm \left(\mp \iota^*\circ \pi_{Y\sslash G}^*\circ\Delta_{Y\sslash G}^*\circ (\pi_{L/ G}^*)^{-1} (o_{i+1,j})\right) + 0 \\
    &  = o_{i+1,j} \pm \left( \mp  (\Delta_{Y\sslash G} \circ \pi_{Y\sslash G} \circ \iota )^*\circ (\pi_{L/ G}^*)^{-1} (o_{i+1,j})\right) \\
    & = o_{i+1,j} \pm \left( \mp \pi_{L/ G}^* \circ (\pi_{L/ G}^*)^{-1} (o_{i+1,j})\right) =0. 
\end{align*}
Let $l\in \N$ and $\alpha \in \Omega^l(Y\sslash G)$. 
Let us denote $f:=\pi_{Y\sslash G}$. 
Then for all $\zeta\in \g$, $p\in Y^-\times Y\sslash G$, $v_2\ldots, v_l\in T_{f(p)} ( Y^-\times Y\sslash G)$,  we have 
\[ (\iota_{\underline{\zeta}} (\pi_{Y\sslash G}^*\alpha))_p(v_2,\ldots, v_l) =\alpha_{f(p)}(df\circ \underline{\zeta}(x) , df_{p}(v_2),\ldots, df_p(v_l))=0, 
\] 
since $ \underline{\zeta}(p_1,p_2) =  (\underline{\zeta}(p_1),0)$ for all $(p_1,p_2)\in Y\times Y\sslash G$. 
Then 
\[ 
\ld_{\underline{\zeta}} (\pi_{Y\sslash G}^*\alpha) = d(\iota_{\underline{\zeta}} \pi_{Y\sslash G}^*\alpha) + \iota_{\underline{\zeta}} (d\pi_{Y\sslash G}^*\alpha) =\iota_{\underline{\zeta}} (\pi_{Y\sslash G}^*d\alpha)=0.    \]  
This shows that $\bulk_{i+1,j}$ is a basic form and completes the induction. 
By construction, if we let
\[ \bulk = \lim_{i\to \infty} \bulk ^{ (i)}, \quad b =0, \]
then \[ \m_{0}^{\bulk, b} (1)= 0. \] 
 \end{proof}

\section{Proof of the key lemma}
\label{section Proof of the key lemma}

\subsection{\texorpdfstring{$G$}{G}-equivariant Kuranishi data}
\label{section G-equivariant Kuranishi data}
It suffices to prove Lemma \ref{key lemma} on $G$-invariant open subsets of $G$-equivariant Kuranishi charts. 
We refer the reader to \cite{kurbook} Chapter 7 -- Chapter 12 for a more detailed discussion of ordinary\footnote{By ``ordinary" we mean the cases where we do not take the $G$-actions into consideration. } integration along the fibers, and to \cite{FukayaLieGroupoids}, and \cite{xiao2023equivariant} Section 5 for an introduction to $G$-equivariant Kuranishi structures. 

\begin{defn}[$G$-equivariant Kuranishi chart]
\label{Kuranishi chart}
  Let $\mathcal{M}$ be a separable metrizable topological space with a topological action by a compact connected Lie group $G$. 
A \textbf{$G$-equivariant Kuranishi chart} on $\mathcal{M}$ is a quadruple
$\mathcal{U} = (U, \E,  \psi , s )$ satisfying the following. 
    \begin{enumerate}[(a)]
        \item $U$ and $\E$ are oriented smooth effective orbifolds, possibly with corners, each equipped with a smooth $G$-action. 
        \item $\E \xrightarrow{\pi } U $ is a smooth $G$-equivariant orbibundle. 
        \item $s : U  \to \E$ is a smooth $G$-equivariant section of $\pi$. 
        \item $\psi: s^{-1}(0)\to \mathcal{M}$ is a $G$-equivariant continuous map, which is a homeomorphism onto an open subset in $\kur$. 
    \end{enumerate}
    \end{defn}
\begin{defn}[Cartan model for equivariant differential forms]
Let $G$ be a compact connected Lie group acting smoothly on a smooth manifold $M$. Let $\g$ be the Lie algebra of $G$ and $\g^*$ be its dual. 
Let $S(\g^*)$ be the symmetric algebra on $\g^*$. 
Define the \textbf{Cartan model} of $G$-equivariant differential forms on $M$ by the $G$-invariant elements of $\Omega(M)\otimes S(\g^*)$, namely
\[ \Omega_G(M):=(\Omega(M)\otimes S(\g^*))^G.  \]
We identify $\Omega_G(M)$ with the set of equivariant polynomials $\g\to \Omega(M)$. 
The \textbf{Cartan differential} $d_G: \Omega_G^{\bullet}(M) \to \Omega_G^{\bullet+1}(M)$ 
is given by 
\[ (d_G\alpha)(\zeta) = d(\alpha(\zeta)) -\iota_{\underline{\zeta}} \alpha(\zeta), \qquad \forall \alpha\in \Omega_G(M) \quad \forall \zeta\in \g. \]
\end{defn}
We follow \cite{Cieliebak} for the definition of equivariant Thom forms. 
\begin{defn}[Equivariant Thom forms]
\label{definition Equivariant Thom forms}
Let $f: E\to B$ be a $G$-equivariant oriented real vector bundle of rank $d$. 
A $G$-equivariant differential form $\tau  \in \Omega_G^d(E)$ is an \textbf{equivariant Thom form} if the following holds. 
\begin{enumerate}[1)]
    \item $\tau$ is equivariantly closed: $d_G \tau =0$.
    \item $\int_{E_x} \tau =1$ for all $x\in B$, where $E_x = f^{-1}(x)$.  
    \item There exists a $G$-invariant open neighborhood $\mathcal{O}$ of the zero section such that 
    \begin{enumerate}
        \item $\operatorname{supp} \tau \subset \mathcal{O}$, 
        \item $\mathcal{O}\cap E_x$ is convex for all $x\in B$, and
        \item $\mathcal{O}\cap E|_K$ is precompact for all compact set $K\subset B$. 
        \end{enumerate}
\end{enumerate}

\end{defn}
\begin{defn}
[CF-perturbation representative on a $G$-invariant open subset of a Kuranishi chart\footnote{``CF"  stands for ``continuous family".}]
\label{defn CF-perturbation on a suborbifold}
Let $\mathcal{U}=(U, \E , \psi, s)$ be a $G$-equivariant Kuranishi chart and $U_{\mathfrak{r}}\subset U$ be a $G$-invariant open subset of $U$. 
A \textbf{$G$-equivariant CF-perturbation representative of $\mathcal{U}$ on $U_{\mathfrak{r}}$} is a continuous family of data 
\begin{equation}
\label{equation CF perturbation representative}
\mathcal{S}_{\mathfrak{r}} = \{ \mathcal{S}_{\mathfrak{r}}^{\epsilon}= (W_{\mathfrak{r}}\xrightarrow{\nu_{\mathfrak{r}}}U_{\mathfrak{r}},\tau_{\mathfrak{r}}, \s^{\epsilon}_{\mathfrak{r}})\mid \epsilon \in (0,1]\}    
\end{equation} such that the following holds.  
\begin{enumerate}[i)]
    \item $W_{\mathfrak{r}}$ is an effective orbifold with a smooth $G$-action.
    \item $\nu_{\mathfrak{r}}: W_{\mathfrak{r}}\to U_{\mathfrak{r}}$ is a smooth oriented $G$-equivariant orbibundle. 
     \item $\tau_{\mathfrak{r}} \in \Omega_G(W_{\mathfrak{r}})$ is a $G$-equivariant Thom form of $ \nu_{\mathfrak{r}} : W_{\mathfrak{r}}\to U_{\mathfrak{r}}$. 
    \item Let $\nu_{\mathfrak{r}}^*(\restr{\E}{U_{\mathfrak{r}}})\to W_{\mathfrak{r}}$ be the pullback bundle of $\restr{\E}{U_{\mathfrak{r}}} \to U_{\mathfrak{r}}$ via $\nu_{\mathfrak{r}}$ and let 
    $pr_2: \nu_{\mathfrak{r}}^*(\restr{\E}{U_{\mathfrak{r}}}) \to \restr{\E}{U_{\mathfrak{r}}}$ be the projection map.  
    $\forall 0< \epsilon\leq 1$, let $\tilde{\s}_{\mathfrak{r}}^{\epsilon}: W_{\mathfrak{r}} \to \nu_{\mathfrak{r}}^*(\restr{\E}{U_{\mathfrak{r}}})$ be a section of the bundle $\nu_{\mathfrak{r}}^*(\restr{\E}{U_{\mathfrak{r}}}) \to W_{\mathfrak{r}}$ satisfying the following. 
    \begin{enumerate}[a)]
        \item $\s_{\mathfrak{r}}^{\epsilon} = pr_2\circ \tilde{\s}_{\mathfrak{r}}^{\epsilon} : W_{\mathfrak{r}}\to \restr{\E}{U_{\mathfrak{r}}}$ is a $G$-equivariant bundle map 
        and the family 
        $\{\s_{\mathfrak{r}}^{\epsilon}\}_{\epsilon \in (0,1]}$ depends smoothly on $\epsilon$. 
        \item  Moreover, $\lim\limits_{\epsilon\to 0}\s_{\mathfrak{r}}^{\epsilon} = s\circ \nu_{\mathfrak{r}}$
    in the compact $C^1$-topology. 
    \[
    \xymatrix{
    \nu_{\mathfrak{r}}^* (\restr{\E}{U_{\mathfrak{r}}})
    \ar[d]_{pr_1}
    \ar[r]^{pr_2}
    & 
     \restr{\E}{U_{\mathfrak{r}}}
     \ar[d]^{\restr{\pi}{ U_{\mathfrak{r}}}}
     \\
    W_{\mathfrak{r}}
    \ar[r]_{ \nu_{\mathfrak{r}}}
    & 
    U_{\mathfrak{r}}
    }
    \]
    \end{enumerate} 
\end{enumerate}
\end{defn}
Consider $\tau_{\mathfrak{r}} : \g \to \Omega(W_{\mathfrak{r}})$ as an equivariant polynomial map.  
If we forget the $G$-actions and replace $\tau_{\mathfrak{r}}$ by its degree zero part $\tau_{\mathfrak{r}}^{(0)}$, 
then a $G$-equivariant CF-perturbation representative of $\mathcal{U}$ of $U_{\mathfrak{r}}$ also induces an ordinary CF-perturbation representative of $\mathcal{U}$. 

\begin{defn}
Let  $\mathcal{U}=(U, \E , \psi, s)$ be a $G$-equivariant Kuranishi chart and let 
   \[ \mathcal{S}_{\rindex} = \{ \mathcal{S}_{\rindex}^{\epsilon} = (W_{\rindex} \xrightarrow{\nu_{\rindex}} U_{\rindex}, \tau_{\rindex} , \s_{\rindex}^{\epsilon})\mid \epsilon \in (0,1]\} \] 
   be a CF-perturbation on a $G$-invariant open subset $U_{\rindex} \subset U$. 
\begin{enumerate}[i)]
    \item \label{local CF-perturbation transverse to zero}
   $\mathcal{S}_{\rindex}$ is said to be \textbf{transverse to zero} if, 
    $\forall 0<\epsilon\leq 1$, the map $\restr{\s_{\rindex}^{\epsilon}}{W_{\rindex}^{\epsilon}}$ is transverse to the zero section on some $G$-invariant neighborhood
    $ W_{\rindex}^{\epsilon}\subset W_{\rindex}$ of the support of $\tau_{\rindex}$. 
     \item \label{strongly submersive with respect to a local CF-perturbation} 
    Let $L$ be a smooth manifold. A $G$-equivariant smooth map $f_{\rindex}: U_{\rindex}\to L$ is said to be \textbf{strongly submersive} with respect to $\mathcal{S}_{\rindex}$ if $\mathcal{S}_{\rindex}$ is transverse to zero and, 
    $\forall 0<\epsilon\leq 1$, 
    , the map 
    \[ \restr{f_{\rindex}\circ \nu_{\rindex}}{(\s_{\mathfrak{r}}^{\epsilon})^{-1}(0)}: (\s_{\mathfrak{r}}^{\epsilon})^{-1}(0) \to L \] 
    is a submersion on some $G$-invariant neighborhood
    $ W_{\rindex}^{\epsilon}\subset W_{\rindex}$ of the support of $\tau_{\rindex}$. 
\end{enumerate}
\end{defn}
\begin{defn}[Ordinary pullback map on a Kuranishi chart] 
Let $\mathcal{U}=(U, \E , \psi, s)$ be a Kuranishi chart and let $f: U \to L$ be a strongly smooth map to a smooth manifold. Then the \textbf{pullback} of $f$ on the Kuranishi chart is the usual pullback of differential forms. 
\end{defn}

\begin{defn}[Ordinary integration along the fiber via a CF-perturbation representative on a suborbifold]
Let $\mathcal{U}=(U, \E , \psi, s)$ be a Kuranishi chart and $U_{\mathfrak{r}}\subset U$ be a suborbifold of $U$. 
Let 
\[ \mathcal{S}_{\mathfrak{r}} = \{ \mathcal{S}_{\mathfrak{r}}^{\epsilon}= (W_{\mathfrak{r}}\xrightarrow{\nu_{\mathfrak{r}}}U_{\mathfrak{r}},\tau_{\mathfrak{r}}, \s^{\epsilon}_{\mathfrak{r}})\mid \epsilon \in (0,1]\}  \]
be CF-perturbation representative of $\mathcal{U}$ on $U_{\mathfrak{r}}$ as in Definition \ref{defn CF-perturbation on a suborbifold}, except that we forget the $G$-actions. 
Let $\mathcal{S}_{\mathfrak{r}}$ be transverse to zero. 
Let $f: U\to L$ be a strongly smooth map such that $f\mid_{U_{\mathfrak{r}}}: U_{\mathfrak{r}}\to L$ is strongly submersive with respect to $\mathcal{S}_{\mathfrak{r}}$. 
The \textbf{integration along the fiber of $f$ via $\mathcal{S}_{\mathfrak{r}}$} is defined as follows. 
Suppose $h \in \Omega_{c}^l(U_{\mathfrak{r}})$ is a compactly supported differential form and $0<\epsilon\leq 1$. 
We have the following data. 
\[
\xymatrix{
& \E_{\mathfrak{r}} \ar[d] & & \\
(\mathfrak{s}_{\mathfrak{r}}^{\epsilon})^{-1}(0)\ar[ur]^{\mathfrak{s}^{\epsilon} }\subset W_{\mathfrak{r}}^{\epsilon}\ar[r]^{\kern 20pt{\nu_{\epsilon}}}   & U_{\mathfrak{r}} \ar[r]^{f\eval_{U_{\mathfrak{r}}}} & L 
}
\]
  Define $f_{!} \left(h; \mathcal{S}^{\epsilon}\right)\in \Omega_{c}^{l - (\dim U_{\mathfrak{r}} - \rank \E_{\mathfrak{r}} - \dim L)}(L)$ to be the unique form satisfying 
  \[ 
 \int_L (f_{U_{\mathfrak{r}}})_!(h; \mathcal{S}_{\mathfrak{r}}^{\epsilon})\wedge \rho 
 =  \int_{(\s_{\mathfrak{r}}^{\epsilon})^{-1}(0)} \nu_{\epsilon}^*
 h \wedge (f_{U_{\mathfrak{r}}}\circ \nu_{\epsilon})^* \rho \wedge \tau_{\mathfrak{r}} 
    \] 
for all $\rho \in \Omega(L)$. 
\end{defn}

\subsection{Existence of basic Thom forms for locally free actions}
\label{section Existence of basic Thom forms for locally free actions}
Let $M$ be a smooth manifold on which a compact connected Lie gorup $G$ acts smoothly.
Let $\g$ be the Lie algebra of $G$.  

\begin{defn}[Locally free and free group actions]
    A smooth $G$-action on a smooth manifold $M$ by a compact Lie group $G$ is \textbf{locally free} if the isotropy group $G_p $ is finite for all  $p\in M$, and it is \textbf{free} if $G_p $ is trivial for all $p\in M$. 
\end{defn} 

Apparently, free group actions are locally free. Moreover, we have the following. 
 
\begin{lem}
\label{Lemma locally free iff g-action is free}
Let a compact connected Lie group $G$ act smoothly on a smooth manifold $M$. Then 
\begin{equation}
\label{equation Lie algebra of isotropy group}
  \operatorname{Lie}(G_p) =\g_p \qquad \forall p\in M.   
\end{equation}
\end{lem}
Thus, the group action is locally free if and only if 
    \[ \g_p := \{\zeta\in \g\mid \underline{\zeta}(p)=0\} = \{0\} \qquad \forall p\in M.  \]

The definition of a connection on a principal $G$-bundle can be generalized as follows. 
\begin{defn}[$G$-connection]
\label{defn G-connection}
A \textbf{$G$-connection} on the de Rham complex $\Omega(M)$ of a $G$-manifold $M$ is a $G$-equivariant linear map $A : \g^* \to \Omega^1(M)$ such that, for all $\zeta\in \g$ and all $\xi\in \g^*$, 
\begin{equation}
\label{equation condition C}
 \iota_{\underline{\zeta}}(A(\xi)) = \pair{\xi,\zeta},    
\end{equation}
where $\pair{\cdot, \cdot}$ denotes the pairing between $\g^*$ and $\g$. 
\end{defn}

The ``$\Rightarrow$" direction of the following proposition is essential in the proof of Lemma \ref{key lemma}. It is also fundamental in the formulation of equivariant de Rham theory in terms of the Weil model and in the generalization of the Chern-Weil theory. 
Due to the following equivalence, a $\g$-differential graded algebra, which the de Rham complex of a $G$-manifold is an example of, with a $G$-connection is often called a ``locally free $\g$-dga". 
\begin{prop}
\label{prop locally free action is equivalent to existence of G-connection}
A smooth action on a smooth manifold $M$ by a compact connected Lie group $G$ is locally free if and only if 
 $\Omega(M)$ has a $G$-connection.
\end{prop}

\begin{proof}
We reproduce the proof in \cite{GS} \S 2.3.4. 
If a $G$-action is locally free, by Lemma \ref{Lemma locally free iff g-action is free}, for each $p\in M$, 
there is an injective homomorphism 
\begin{equation}\label{equation g is a subspace of TpM}
\g \to T_pM, \qquad \zeta\mapsto \underline{\zeta}(p).    
\end{equation} 
Fix a basis $\zeta_1,\ldots, \zeta_r$ for $\g$, a dual basis $\theta_1,\ldots, \theta_r$ for $\g^*$, and a $G$-invariant metric $g$ on $M$ such that, for all $p\in M$, the vectors $\underline{\zeta_1}(p), \ldots, \underline{\zeta_r}(p)$ are orthonormal. 
Then we can define $A: \g^*\to \Omega^1(M)$ by 
\begin{equation}\label{exists G-connection}
A(\theta_i) = g(\underline{\zeta_i}, -) \quad \forall 1\leq i \leq r.    
\end{equation} 
Then the one-forms $\Theta_i =A(\theta_i)$, $1\leq i \leq r$, span the vertical subbundle $V$ of $T^*M$. 
Let $H = V^{\perp}$ be the horizontal subbundle.  

Since
\[ \iota_{\underline{\zeta_j}}(A(\theta_i)) = g(\underline{\zeta_i},\underline{\zeta_j}) = \delta_{ij} = \pair{\theta_i,\zeta_j} \quad \forall i,j, \]
the condition \eqref{equation condition C} is satisfied. 
We now show the $G$-equivariance 
\begin{equation}
    \label{equation G-equivariance of the connection}
    \mathcal{L}_{\zeta}A(\xi) = A(\operatorname{ad}_{\zeta}^*(\xi)).  
\end{equation} 
Let $k,i\in \{1\ldots, r\}$. 
For any $j\in \{1\ldots, r\}$,   $\iota_{\underline{{\zeta}_j}}(A(\theta_i))$ is constant. 
Thus, 
\begin{align}
\label{equation connection is equivariant}
    0 
    & = \mathcal{L}_{\underline{\zeta_k}}\iota_{\underline{\zeta_j}}(A(\theta_i)) 
   \nonumber  \\ 
    & = 
    \left [ \mathcal{L}_{\underline{\zeta_k}}, 
    \iota_{\underline{\zeta_j}}
    \right ](A(\theta_i)) 
    + 
    \iota_{\underline{\zeta_j}} 
    \mathcal{L}_{\underline{\zeta_k} }(A(\theta_i))  \nonumber
    \\
    & = \iota_{[\underline{\zeta_k}, \underline{\zeta_j}]}(A(\theta_i)) + 
    \iota_{\underline{\zeta_j}} 
    \mathcal{L}_{\underline{\zeta_k} }(A(\theta_i)). 
\end{align}
Here $ \left [ \mathcal{L}_{\underline{\zeta_k}}, 
    \iota_{\underline{\zeta_j}}
    \right ]$ is the commutator $\mathcal{L}_{\underline{\zeta_k}} 
    \iota_{\underline{\zeta_j}} -  
    \iota_{\underline{\zeta_j}}\mathcal{L}_{\underline{\zeta_k}}$. 
Let the $c_{kj}^i$ be the structure constants defined by 
\[ [\underline{\zeta_k},\underline{\zeta_j}] = \sum_{i=1}^r c_{kj}^i \underline{\zeta_i}. \]
By \eqref{equation connection is equivariant}, we have
\[ \mathcal{L}_{\underline{\zeta_k} }(A(\theta_i)) = -\sum_{j=1}^r c_{kj}^i \Theta_j + \alpha_{ki}, \]
for some horizontal $\alpha_{ki}$. 
Since the metric is invariant, and both $\mathcal{L}_{\underline{\zeta_k} }(A(\theta_i))$ and $-\sum_{j=1}^r c_{kj}^i \Theta_j$ are vertical, we have $\alpha_{ki}=0$. 
On the other hand, for any $j$, we have
\[ \pair{\operatorname{ad}_{\zeta_k}^*(\theta_i), \zeta_j} = \pair{\theta_i,- [\zeta_k,\zeta_j]} = -c_{kj}^i. \]
This implies that $\operatorname{ad}_{\zeta_k}^*(\theta_i) = -\sum_{j=1}^r c_{kj}^i\zeta_j$. 
Thus, 
\[ 
A( \operatorname{ad}_{\zeta_k}^*(\theta_i) ) = - \sum_{j=1}^r c_{kj}^i g(\underline{\zeta_j}, -) = - \sum_{j=1}^r c_{kj}^i \Theta_j = \mathcal{L}_{\underline{\zeta_k}}(A(\theta_i)). 
\]
Therefore, $A$ is a $G$-connection on $\Omega(M)$. 

Conversely, if there exists a $G$-connection $A$ on $\Omega(M)$, then  
$\g_p=\{0\}$ for all $p\in M$, and thus the $G$-action is locally free by \eqref{equation Lie algebra of isotropy group}.     
\end{proof}

Consider the Cartan operator defined in the following theorem, whose proof is omitted. 
\begin{thm}[Cartan operator is homotopic to identity, \cite{GS} \S 5]  
\label{thm Cartan operator is homotopic to identity}
Consider a locally free action of a compact connected Lie group $G$ of dimension $r$ on a smooth manifold $M$. Let $\g$ be the Lie algebra of $G$ and $\g^*$ be its dual. Let $S(\g^*)$ be the symmetric algebra on $\g^*$. Then we can equip $\Omega(M)$ with a $G$-connection $A$. 
Let $F^A$ be the associated curvature, and let $A_i=A(\theta_i)$ for all $1\leq i \leq r$
as in Proposition \ref{prop locally free action is equivalent to existence of G-connection}.  
There exists an operator $\operatorname{Car}^A$, defined by the composition 
\begin{equation}
    \label{equation Cartan operator}
     \operatorname{Car}^A: \left( \Omega(M)\otimes S(\g^*) \right)^G \xrightarrow{\operatorname{Hor}^{A}} \left( \Omega_{hor}^A(M)\otimes S(\g^*) \right)^G \xrightarrow{1\otimes \kappa_{S(\g^*)}} \Omega_{bas}(M),
\end{equation} 
where $\operatorname{Hor}^{A}$ 
is the projection to 
$\left( \Omega_{hor}^A(M)\otimes S(\g^*) \right)^G $, and $ \kappa_{S(\g^*)}: S(\g^*)\to \Omega(M)$ is defined by 
\begin{equation}
\label{equation Chern-Weil on polynomials}
\kappa_{S(\g^*)}(\gamma_1 \cdots \gamma_k) = (\gamma_1\circ F^{A}) \wedge \cdots \wedge(\gamma_k\circ F^{A}) \quad \forall \gamma_1, \ldots, \gamma_k\in \g^*,      
\end{equation} 
where $\gamma_i\circ F^A= \left(\Omega^2(M)^*\xrightarrow{(F^{A})^*} \g \xrightarrow{\gamma_i}\R\right)$. 
The map \[ \operatorname{Car}^A: (\Omega_G(M),d_G)\to (\Omega_{bas}(M), d) \]is a chain map which is chain homotopic to the identity. 
\end{thm}

The construction in Theorem \ref{thm Cartan operator is homotopic to identity} is used to show 
Theorem \ref{thm Existence of $G$-basic Thom forms} in \cite{Cieliebak}. We refer the reader to \cite{Cieliebak} for the proof of it. 

\begin{thm}[Existence of basic Thom forms \cite{Cieliebak} Theorem 3.8 and Remark 5.2]
\label{thm Existence of $G$-basic Thom forms}
Let $\tau \in \Omega_G^d(E)$ be an equivariant Thom form on the $G$-equivariant oriented real vector bundle $E\to B$  of rank $d$. 
Suppose the $G$-actions on $E$ and $B$ are locally free. 
Then there exists a $G$-connection $A:\g^*\to \Omega^1(E)$ such the Cartan operator $\operatorname{Car}^A$ as in \eqref{equation Cartan operator} carries $\tau$ to a $G$-basic Thom form $\tau_A$, which also satisfies Definition \ref{definition Equivariant Thom forms}.  
\end{thm}
We can generalize the construction to define basic Thom forms on $G$-equivariant oriented orbibundles with free actions.

\subsection{Proof of Lemma \ref{key lemma}}
\label{subsection Proof of key lemma }
\begin{proof}[Proof of Lemma \ref{key lemma}]

Consider a $G$-equivariant Kuranishi structure on $\Mq_{k+1,l}(L,J,\beta)$ which satisfy Assumption \ref{assumptions}. 

Recall that pullback maps defined via strongly smooth maps of the form $\mathcal{M}\to L$, from a Kuranishi space to a smooth manifold $L$, are defined to be chart-wise pullback, and the integration along the fiber maps defined via strongly smooth maps of the form $\mathcal{M}\to L$,  which are strongly submersive with respect to a CF-perturbation, are defined by taking integration along the fiber maps on suborbifolds that cover the Kuranishi charts and gluing by partitions of unity. 
Therefore, it suffices to show that 
the pullback maps of the form $\evaluation_{i,(k+1,l,\beta)}^*$, $(\evaluation_{(k+1,l,\beta)}^j)^*$, $\iota^*$, 
preserve $G$-basicness on the $G$-equivariant Kuranishi charts,
and that the integration along the fiber maps of the form $\evaluation_{0,(k+1,l,\beta)!}$ preserve $G$-basicness on $G$-invariant open subsets of the $G$-equivariant Kuranishi charts. 

Since $\iota$ is a $G$-equivariant map of smooth manifolds, $\iota^*$ commutes with $\mathcal{L}_{\underline{\zeta}}$, $\iota_{\underline{\zeta}}$ for all $\zeta\in \g$, showing that $\iota^*$ preserves $G$-basicness. 
Similarly, if $\Mq_{k+1,l}(L,J,\beta)$ satisfies Assumption \ref{assumptions}, then pulling back by the equivariant maps of the form $\evaluation_{i,(k+1,l,\beta)}$,  $ \evaluation_{(k+1,l,\beta)}^j $ also preserve $G$-basicness.

Let $\mathcal{U}=(U, \E , \psi, s)$ be a $G$-equivariant Kuranishi chart of the moduli space $\Mq_{k+1,l}(L,J,\beta)$. 
Let   
\[ \mathcal{S}_{\rindex} = \{ \mathcal{S}_{\rindex}^{\epsilon} = (W_{\rindex} \xrightarrow{\nu_{\rindex}} U_{\rindex}, \tau_{\rindex} , \s_{\rindex}^{\epsilon})\mid \epsilon \in (0,1]\} \] 
   be a CF-perturbation representative on a nonempty $G$-invariant open subset $U_{\rindex} \subset U$.
Let $f_{U_{\mathfrak{r}}}$ denote the restriction of $\evaluation_{0,(k+1,l,\beta)}$ to $U_{\mathfrak{r}}$. 
By assumption, it is a $G$-equivariant strongly smooth map which is strongly submersive.   
By Asssumption \ref{assumptions} and Theorem \ref{thm Existence of $G$-basic Thom forms}, we may assume that the equivariant Thom form  $\tau_{\mathfrak{r}}$ is  $G$-basic.  

We want to show that $(f_{U_{\mathfrak{r}}})_!$ commutes with $\iota_{\underline{\zeta}}$ and $\mathcal{L}_{\underline{\zeta}}$ for all $\zeta\in \g$.  

Then for all 
$\rho \in  \Omega (L) $ 
and all $\zeta\in \g$ we have
\begin{align*}
   &  \int_L (f_{U_{\mathfrak{r}}})_!(\iota_{\underline{\zeta}}h; \mathcal{S}_{\rindex}^{\epsilon})\wedge \rho \\
     = &  \int_{(\s_{\mathfrak{r}}^{\epsilon})^{-1}(0)} \nu_{\epsilon}^*\iota_{\underline{\zeta}}h \wedge (f_{U_{\mathfrak{r}}}\circ \nu_{\epsilon})^* \rho \wedge \tau_{\mathfrak{r}} 
    \\
    = & \int_{(\s_{\mathfrak{r}}^{\epsilon})^{-1}(0)} \iota_{\underline{\zeta}}\nu_{\epsilon}^* h \wedge (f_{U_{\mathfrak{r}}}\circ \nu_{\epsilon})^* \rho \wedge \tau_{\mathfrak{r}}  
   \\
   = &  (-1)^{\deg h +1}\int_{(\s_{\mathfrak{r}}^{\epsilon})^{-1}(0)} \nu_{\epsilon}^* h \wedge \iota_{\underline{\zeta}}(f_{U_{\mathfrak{r}}}\circ \nu_{\epsilon})^* \rho \wedge \tau_{\mathfrak{r}} 
   \\
   & +  (-1)^{\deg h + \deg \rho +1}\int_{(\s_{\mathfrak{r}}^{\epsilon})^{-1}(0)} \nu_{\epsilon}^* h \wedge (f_{U_{\mathfrak{r}}}\circ \nu_{\epsilon})^* \rho \wedge \iota_{\underline{\zeta}}\tau_{\mathfrak{r}} 
   \\
  = &  (-1)^{\deg h +1}\int_{(\s_{\mathfrak{r}}^{\epsilon})^{-1}(0)} \nu_{\epsilon}^* h \wedge (f_{U_{\mathfrak{r}}}\circ \nu_{\epsilon})^* \iota_{\underline{\zeta}}\rho \wedge \tau_{\mathfrak{r}} \\
   = &  (-1)^{\deg h +1}\int_{L} (f_{U_{\mathfrak{r}}})_!h\wedge\iota_{\underline{\zeta}}  \rho  \\
   = &   \int_{L} \iota_{\underline{\zeta}} (f_{U_{\mathfrak{r}}})_!(h;\mathcal{S}_{\rindex}^{\epsilon})\wedge \rho. 
\end{align*}
Thus, 
\[ (f_{U_{\mathfrak{r}}})_!(\iota_{\underline{\zeta}}h;\mathcal{S}_{\rindex}^{\epsilon}) = \iota_{\underline{\zeta}} (f_{U_{\mathfrak{r}}})_!(h;\mathcal{S}_{\rindex}^{\epsilon}). \] 
Similarly, 
\begin{align*}
   \int_L (f_{U_{\mathfrak{r}}})_!(dh;\mathcal{S}_{\rindex}^{\epsilon})\wedge \rho   &  = \int_{(\s_{\mathfrak{r}}^{\epsilon})^{-1}(0)} \nu_{\epsilon}^*dh \wedge (f_{U_{\mathfrak{r}}}\circ \nu_{\epsilon})^* \rho \wedge \tau_{\mathfrak{r}}
    \\
   &  = \int_{(\s_{\mathfrak{r}}^{\epsilon})^{-1}(0)} d \nu_{\epsilon}^* h \wedge (f_{U_{\mathfrak{r}}}\circ \nu_{\epsilon})^* \rho \wedge \tau_{\mathfrak{r}}
   \\
   &  = (-1)^{\deg h +1}\int_{(\s_{\mathfrak{r}}^{\epsilon})^{-1}(0)} \nu_{\epsilon}^* h \wedge d(f_{U_{\mathfrak{r}}}\circ \nu_{\epsilon})^* \rho \wedge \tau_{\mathfrak{r}}  \\
   & = (-1)^{\deg h +1}\int_{(\s_{\mathfrak{r}}^{\epsilon})^{-1}(0)} \nu_{\epsilon}^* h \wedge (f_{U_{\mathfrak{r}}}\circ \nu_{\epsilon})^* d\rho \wedge \tau_{\mathfrak{r}} \\
   & = (-1)^{\deg h +1}\int_{L} (f_{U_{\mathfrak{r}}})_!(h;\mathcal{S}_{\rindex}^{\epsilon})\wedge d\rho \\
   & =  \int_{L} d(f_{U_{\mathfrak{r}}})_!(h;\mathcal{S}_{\rindex}^{\epsilon})\wedge \rho. 
\end{align*}
Hence,  
\[ (f_{U_{\mathfrak{r}}})_!(dh;\mathcal{S}_{\rindex}^{\epsilon}) = d(f_{U_{\mathfrak{r}}})_!(h;\mathcal{S}_{\rindex}^{\epsilon}). \]
Moreover, we have 
\[ (f_{U_{\mathfrak{r}}})_!(\mathcal{L}_{\underline{\zeta}} h;\mathcal{S}_{\rindex}^{\epsilon})  = (f_{U_{\mathfrak{r}}})_!( \iota_{\underline{\zeta}}dh + d\iota_{\underline{\zeta}} h ;\mathcal{S}_{\rindex}^{\epsilon}) = (\iota_{\underline{\zeta}}d + d\iota_{\underline{\zeta}} )(f_{U_{\mathfrak{r}}})_!(h;\mathcal{S}_{\rindex}^{\epsilon}) = \mathcal{L}_{\underline{\zeta}}(f_{U_{\mathfrak{r}}})_!(h;\mathcal{S}_{\rindex}^{\epsilon}) \]
by Cartan's magic formula. 
Thus, if $h$ is a basic form on $U_{\mathfrak{r}}$, then $(f_{U_{\mathfrak{r}}})_!(h;\mathcal{S}_{\rindex}^{\epsilon}) $ is a basic form on $L$. 
\end{proof}

\section{A conjecture}
\label{section Conjecture on the quantum Kirwan map}
We state a related conjecture. 
\begin{conj}
\label{Conjecture on the quantum Kirwan map} Let $(Y,\omega_Y,G,\mu)$ be a Hamiltonian $G$-manifold such that $G$ acts on $\mu^{-1}(0)$ freely and  $Y\sslash G$ be the corresponding symplectic quotient.
    Then there exists map
    \begin{equation}
       \kappa :  QH_G(Y,\Lambda_{0,nov}) \to QH(Y\sslash G,\Lambda_{0,nov})
    \end{equation}
    such that the equivariant $\Ainf$ algebra $\left(\Omega_G(L,\Lambda_{0,nov}), \{\m_k^G\}_{k\in \N}\right)$ associated to the moment Lagrangian correspondence $L\subset Y^-\times Y\sslash G$ satisfies
    \[ (\m_0^G)^{ \bulk\otimes\kappa(\bulk) , b=0} (1)=0 \qquad \forall \bulk\in QH_G(Y,\Lambda_{0,nov}). 
    \]
\end{conj}
We expect $\kappa$ to be similar to the quantum Kirwan maps constructed in \cite{WoodwardXu}. (Also see \cite{QuantumKirwanI}, \cite{QuantumKirwanII}, and \cite{QuantumKirwanIII}.) 

\printbibliography

\Addresses 
\end{document}